\documentclass[reqno]{amsart}
              
\usepackage{courier}
\usepackage{cite}
\usepackage{verbatim,curves,graphics}
\usepackage{lscape}
\usepackage[dvips]{graphicx}
\usepackage{ifthen}
\usepackage{amsmath}
\usepackage{amsfonts,latexsym,amssymb}
\usepackage{bm}
\usepackage{dsfont}
\usepackage{mathrsfs}
\usepackage{color}
\usepackage{xypic}
\usepackage{ae}
\usepackage{aecompl}

\newtheoremstyle{personal}%
{12pt}
{12pt}
{\slshape}
{}
{\bfseries}
{.}
{.5em}
{}
\newtheoremstyle{personalB}%
{12pt}
{12pt}
{\slshape}
{}
{\bfseries}
{.}
{\newline}
{}
\theoremstyle{personal}%
\newtheorem{thm}{Theorem}[section]

\newtheorem{lem}[thm]{Lemma}

\theoremstyle{definition}

\newtheorem{rem}{Remark}[section]
\theoremstyle{personalB}%

\numberwithin{equation}{section}
\DeclareMathOperator*{\toup}{\longrightarrow} 

\newcommand{\N}{\mathds{N}}
\newcommand{\M}{\mathds{M}}
\newcommand{\Z}{\mathds{Z}}
\newcommand{\R}{\mathds{R}}

\newcommand{\crit}{\mathrm{crit}}
\newcommand{\ttau}{{\bm{\tau}}}

\newcommand{\ev}{\mathrm{ev}}
\newcommand{\dist}{\mathrm{dist}}

\newcommand{\injrad}{\mathrm{injrad}}
\newcommand{\id}{\mathrm{id}}
\newcommand{\ind}{\mathrm{ind}}
\newcommand{\nul}{\mathrm{nul}}

\newcommand{\diff}{\mathrm{d}}
\newcommand{\fix}{\mathrm{fix}}

\newcommand{\Hess}{\mathrm{Hess}}
\newcommand{\Tan}{\mathrm{T}}
\newcommand{\orb}{\mathrm{orb}}

\begin{document}

\title[Isometry-invariant geodesics and the fundamental group, II]{Isometry-invariant geodesics\\ and the fundamental group, II}

\author{Leonardo Macarini}
\address{Universidade Federal do Rio de Janeiro, Instituto de Matem\'atica\newline\indent Cidade Universit\'aria, CEP 21941-909, Rio de Janeiro, Brazil}
\email{leonardo@impa.br}

\author{Marco Mazzucchelli}
\address{CNRS and \'Ecole Normale Sup\'erieure de Lyon, UMPA\newline\indent  69364 Lyon Cedex 07, France}
\email{marco.mazzucchelli@ens-lyon.fr}


\date{April 22, 2015. \emph{Revised}: February 15, 2016.}

\begin{abstract}
We show that on a closed Riemannian manifold with fundamental group isomorphic to $\Z$, other than the circle, every isometry that is homotopic to the identity possesses infinitely many invariant geodesics. This completes a recent result in \cite{Mazzucchelli:Isometry_invariant_geodesics_and_the_fundamental_group} of the second author.

\tableofcontents
\end{abstract}

\maketitle

\section{Introduction}\label{s:Introduction}

In this paper, we complete the study began in \cite{Mazzucchelli:Isometry_invariant_geodesics_and_the_fundamental_group} of isometry-invariant geodesics on closed Riemannian manifolds with infinite abelian fundamental group. Isometry-invariant geodesics play the role of closed geodesics in a Riemannian setting with symmetry. Given an isometry $I$ of a closed connected Riemannian manifold $(M,g)$, a geodesic $\gamma:\R\looparrowright M$ is called $I$-invariant if $I(\gamma(t))=\gamma(t+\tau)$ for some positive $\tau>0$ and for all $t\in\R$. Intuitively, these curves should be the closed geodesics of the possibly singular quotient $M/I$. 

The study of isometry-invariant geodesics was initiated by Grove \cite{Grove:Condition_C_for_the_energy_integral_on_certain_path_spaces_and_applications_to_the_theory_of_geodesics, Grove:Isometry_invariant_geodesics} in the 1970s. The problem admits a variational description, which generalizes the one of closed geodesics: isometry-invariant geodesics are the critical points of an energy function defined on a space of invariant paths. If the considered isometry is homotopic to the identity, this space of invariant paths is homotopy equivalent to the free loop space. This may induce someone to naively conjecture that all multiplicity results for closed geodesics remain true for isometry-invariant geodesics, provided the isometry is homotopic to the identity. A quite sophisticated argument due to Grove and Tanaka \cite{Grove_Tanaka:On_the_number_of_invariant_closed_geodesics_BULLETTIN, Grove_Tanaka:On_the_number_of_invariant_closed_geodesics_ACTA, Grove:The_isometry_invariant_geodesics_problem_Closed_and_open} shows that this is the case for Gromoll and Meyer's theorem: every closed Riemannian manifold with non-monogenic rational cohomology admits infinitely many isometry-invariant geodesics. This result is proved by cleverly exploiting the richness of the homology of the free loop space. However, there are multiplicity results, such as the existence of infinitely many closed geodesics on Riemannian 2-spheres \cite{Bangert:On_the_existence_of_closed_geodesics_on_two_spheres, Franks:Geodesics_on_S2_and_periodic_points_of_annulus_homeomorphisms, Hingston:On_the_growth_of_the_number_of_closed_geodesics_on_the_two_sphere}, whose proofs need arguments that go beyond the abundance of the homology of the free loop space. These results may fail for isometry-invariant geodesics: for instance, a non-trivial rotation on a round $2$-sphere has only one invariant geodesic.

A famous theorem of Bangert and Hingston implies that closed Riemannian manifolds with infinite abelian fundamental group always possess infinitely many closed geodesics. As in the case of the 2-sphere, the proof of this result combines general minimax techniques from Morse theory with the investigation of homotopy groups of the free loop space and more ad-hoc arguments. In a previous paper of the second author \cite{Mazzucchelli:Isometry_invariant_geodesics_and_the_fundamental_group}, Bangert and Hingston's theorem was extended to the isometry-invariant setting under the extra assumption that the infinite abelian fundamental group was not cyclic. The main result of this paper completes the extension of the result.

\begin{thm}\label{t:pi_1_Z}
Let $(M,g)$ be a closed connected Riemannian manifold different from the circle and with fundamental group isomorphic to $\Z$. Every isometry of $(M,g)$ that is homotopic to the identity possesses infinitely many invariant geodesics.
\end{thm}

The proof of this result is not a mere generalization of Bangert and Hingston's one, but requires crucial new ingredients. Actually, Theorem~\ref{t:pi_1_Z} will be a corollary of a more general result, Theorem~\ref{t:main}. This latter statement asserts the existence of infinitely many isometry-invariant geodesics, provided the space of invariant curves has infinitely many connected components with enough non-trivial homotopy in a fixed positive degree. The proof will be based on  the technical Lemma~\ref{l:iterated_mountain_passes} asserting that, for any positive degree $d$, a sufficiently iterated periodic invariant geodesic is not a $d$-dimensional mountain pass. We believe that this lemma has independent interest, and for instance a version of it for the Lagrangian free-period action functional (see \cite{Abbondandolo:Lectures_on_the_free_period_Lagrangian_action_functional}) might find application to the multiplicity problem for periodic orbits of Tonelli Lagrangian systems with prescribed energy.

\subsection{Organization of the paper}
In Section~\ref{s:Preliminaries}, we recall the variational setting of the energy function for isometry-invariant geodesics, and we quote the results of Grove and Tanaka that describe the Morse theoretic properties of iterated periodic isometry-invariant geodesics. In Section~\ref{s:broken_geodesics} we introduce a finite dimensional reduction of the space of invariant curves by means of Morse's broken geodesics approximations, and we focus on the  properties of this reduction that will be needed later on. Section~\ref{s:main_lemma}, which is the core of the paper, is devoted to the proof of the above mentioned main technical lemma. Finally, in Section~\ref{s:main_theorem}, we prove the main result of the paper, Theorem~\ref{t:main}, and carry over the proof of Theorem~\ref{t:pi_1_Z} as a corollary of it.

\subsection{Acknowledgements} 
This project began during a research stay of the authors at IMPA (Rio de Janeiro, Brazil). The authors thank the Brazilian-French Network in Mathematics, as well as Henrique Bursztyn, for providing financial support, and IMPA for its  stimulating working environment. The second author also acknowledges support by the ANR projects WKBHJ (ANR-12-BS01-0020) and COSPIN (ANR-13-JS01-0008-01).

\section{Preliminaries}
\label{s:Preliminaries}

\subsection{The space of isometry-invariant curves}
Let $I$ be an isometry of a closed connected Riemannian manifold $(M,g)$. We consider the following space of curves on which $I$ acts as a translation of time $\tau>0$
\[
\Lambda^\tau(M;I):=\big\{\zeta\in  W^{1,2}_{\mathrm{loc}}(\R;M)\ \big|\ I(\zeta(t))=\zeta(t+\tau)\ \forall t\in\R \big\}.
\]
We denote by $E^\tau:\Lambda^\tau(M;I)\to\R$ the energy function
\[
E^\tau(\zeta)=\frac1\tau \int_0^\tau g_{\zeta(t)}(\dot\zeta(t),\dot\zeta(t))\,\diff t,
\]
whose critical points are precisely those smooth geodesics of $(M,g)$ that are contained in $\Lambda^\tau(M;I)$, that is, the $I$-invariant geodesics of $(M,g)$ with a suitable parametrization. It is well known that, with the usual $W^{1,2}$ Riemannian metric on path spaces, $\Lambda^\tau(M;I)$ is a complete Hilbert-Riemannian manifold, and $E^\tau$ satisfies the Palais-Smale condition. We refer the reader to \cite{Grove:Condition_C_for_the_energy_integral_on_certain_path_spaces_and_applications_to_the_theory_of_geodesics} for the background on this functional setting. Since any reparametrization with constant speed of a geodesic is still a geodesic, the particular value $\tau$ is not conceptually relevant here. We will often set $\tau=1$ and, in such case, omit $\tau$ from the notation.

The real line $\R$ acts on $\Lambda(M;I)$ by translation, i.e.\ $t\cdot \zeta=\zeta(t+\cdot)$ for all $t\in\R$ and $\zeta\in\Lambda(M;I)$. The energy function $E$ is invariant by this action, and therefore its critical points come in orbits. An open $I$-invariant geodesic corresponds to a unique critical orbit of $E$, whereas a $p$-periodic $I$-invariant geodesic $\gamma$ with positive energy corresponds to a countable sequence of embedded critical circles $\{\orb(\gamma^{mp+1})\ |\ m\in\N\}$, where $\gamma^{mp+1}(t)=\gamma((mp+1)t)$. Notice that there may also be uninteresting critical points of $E$: the fixed points of $I$, which come in totally geodesic submanifolds of $(M,g)$. As critical points of $E$, these stationary curves have zero critical value and, by compactness, are contained in finitely many connected components of $\Lambda(M;I)$. In this paper, as usual, we will only be interested in $I$-invariant geodesics with positive energy, and we will consider two critical orbits $\orb(\gamma_1)$ and $\orb(\gamma_2)$ of the energy $E$ as distinct $I$-invariant geodesics if and only if $\gamma_1(\R)\neq\gamma_2(\R)$, that is, if and only if they define different immersed submanifolds of $(M,g)$.

\subsection{Periodic isometry-invariant geodesics}

By a classical theorem of Grove \cite[Thm.~2.4]{Grove:Isometry_invariant_geodesics}, the closure of the critical orbit of an open $I$-invariant geodesic contains uncountably many other critical orbits of $E$. Therefore, for the study of multiplicity results, we will always assume that all $I$-invariant geodesics are periodic curves.

Consider the sequence of critical circles $\{\orb(\gamma^{mp+1})\ |\ m\in\N\}$ of the energy function $E$ associated to an $I$-invariant geodesic of minimal period $p>0$. Notice that $p\geq 1$. Indeed, if $p\in(0,1)$, there exists $\tau\in(0,p]$ such that $I(\gamma(t))=\gamma(t+\tau)$ for all $t\in\R$, but this is not possible since the critical circle of the curve $\gamma^\tau(t):=\gamma(\tau t)$ does not belong to $\{\orb(\gamma^{mp+1})\ |\ m\in\N\}$. In the following, we wish to summarize the Morse-theoretic properties of the sequence $\{\orb(\gamma^{mp+1})\ |\ m\in\N\}$.
We denote by $\ind(E,\gamma)$ and $\nul(E,\gamma)$ the Morse index and the nullity of $E$ at $\gamma$. The following lemma was proved by Grove and Tanaka  \cite[Lemma~2.8]{Grove_Tanaka:On_the_number_of_invariant_closed_geodesics_ACTA} under the assumption that the isometry $I$ has finite order, and extended to the general case by Tanaka \cite[Lemma~1.8]{Tanaka:On_the_existence_of_infinitely_many_isometry_invariant_geodesics}.

\begin{lem}[Grove-Tanaka]\label{l:mean_index}
Either $\ind(E,\gamma^{mp+1})=0$ for all non-negative integers $m$, or  $\ind(E,\gamma^{mp+1})\to\infty$ as $m\to\infty$.
\hfill\qed
\end{lem}

Notice that, from the point of view of Morse theory, the critical point $\gamma^{mp+1}$ of $E$ is equivalent to the critical point $\gamma$ of $E^{mp+1}$. Indeed, the diffeomorphism 
\begin{align*}
 \psi^{mp+1}: \Lambda^{mp+1}(M;I)\toup^{\cong}\Lambda(M;I), \qquad
 \psi^{mp+1}(\zeta)=\zeta^{mp+1},
\end{align*}
satisfies $E\circ\psi^{mp+1}=(mp+1)^2 E^{mp+1}$, and in particular 
\[\ind(E, \gamma^{mp+1})=\ind(E^{mp+1},\gamma).\] 
The following lemma follows from the arguments in \cite[Sections~2-3]{Grove_Tanaka:On_the_number_of_invariant_closed_geodesics_ACTA} and \cite[Sections~2-3]{Tanaka:On_the_existence_of_infinitely_many_isometry_invariant_geodesics}. We provide its proof for the reader's convenience.

\begin{lem}[Grove-Tanaka]\label{l:Grove_Tanaka}
There exist finitely many Hilbert manifolds $\Omega_1,...,\Omega_n$, positive real numbers $q_1,...,q_n$ that are multiples of $p$, real numbers  $q_1',...,q_n'$ satisfying $q_i'\in[0,q_i)$ for all $i$, and a partition $\M_1\cup...\cup\M_n$ of the set $\N\setminus\{1,2,...,n_0\}$, for some $n_0\in\N$,  such that for all $m\in\M_i$ the following properties hold:
\begin{itemize}
\item[(i)]  $\Omega_{i}$ is a complete Hilbert submanifold of both $\Lambda^{q_i}(M;\id)$ and $\Lambda^{mp+1}(M;I)$, and is invariant by the gradient flow of $E^{mp+1}$;
\item[(ii)] $E^{q_i}|_{\Omega_i}=E^{mp+1}|_{\Omega_i}$;
\item[(iii)] $\nul(E^{mp+1},\gamma)=\nul(E^{mp+1}|_{\Omega_i},\gamma)$;
\item[(iv)] $mp+1\equiv q_i'\quad \mathrm{mod}\ q_i$.
\end{itemize}
\end{lem}

\begin{rem}
Lemma~\ref{l:Grove_Tanaka} was stated in the previous paper of the second author \cite[Lemma~2.3]{Mazzucchelli:Isometry_invariant_geodesics_and_the_fundamental_group} with two unfortunate typos: the numbers $q_1,...,q_n$ were said to be integers instead of real numbers, and the manifolds $\Omega_i$ were said to be contained in $\Lambda^{q_i}(M;I)$ instead of $\Lambda^{q_i}(M;\id)$. Both typos were rather obvious from the context, since the period $p$ is not necessarily rational, and the lemma was employed correctly in the paper.
\end{rem}

\begin{rem}
If the isometry $I$ is the identity, that is, if we are in the closed geodesics setting, Lemma~\ref{l:Grove_Tanaka} reduces to a classical result due to  Gromoll and Meyer  \cite{Gromoll_Meyer:Periodic_geodesics_on_compact_Riemannian_manifolds}: in this case we have $p=1$, each $q_i$ is a positive integer dividing $m-1$ for all the numbers $m\in\M_i$, the numbers  $q_i'$ are all equal to zero, and the manifolds $\Omega_i$ are simply the free loop spaces $\Lambda^{q_i}(M;\id)$.
\end{rem}

\begin{proof}[Proof of Lemma~\ref{l:Grove_Tanaka}]
We will treat the cases in which the period $p$ is rational and irrational separately. We begin with the rational case: $p=a/b$ for some relatively prime integers $a,b\in\N$. We recall that the fixed point set of an isometry of a closed Riemannian manifold is a disjoint union of finitely many closed totally geodesic submanifolds, see~\cite[page~59]{Kobayashi:Transformation_groups_in_differential_geometry}. Notice that, for all multiples $s$ of the numerator $a$ and for all $m\in\N$, the $I$-invariant geodesic $\gamma$ belongs to the Hilbert submanifold $\Lambda^{mp+1}(\fix(I^s);I)\subset\Lambda^{mp+1}(M;I)$. By \cite[Lemma~2.1]{Tanaka:On_the_existence_of_infinitely_many_isometry_invariant_geodesics}, there is such $s\in\N$ and some $n_0\in\N$ such that 
\begin{align*}
\nul(E^{mp+1},\gamma)=\nul(E^{mp+1}|_{\Lambda^{mp+1}(\fix(I^s);I)},\gamma),\qquad
\forall m\in\N\mbox{ with }m\geq n_0. 
\end{align*}
Let $\nabla E^{mp+1}$ denote the gradient of the energy function $E^{mp+1}:\Lambda^{mp+1}(M;I)\to\R$ with respect to its usual Riemannian metric. The argument in the proof of~\cite[Proposition~3.5]{Grove_Tanaka:On_the_number_of_invariant_closed_geodesics_ACTA} implies that the gradient $\nabla E^{mp+1}(\zeta)$ is tangent to the submanifold $\Lambda^{mp+1}(\fix(I^s);I)$ for all $\zeta\in \Lambda^{mp+1}(\fix(I^s);I)$. The isometry $I$ has order $s$ on $\fix(I^s)$. Therefore, from now on, we can work inside the manifold $\fix(I^s)$ and apply the arguments in \cite[Section~2]{Grove_Tanaka:On_the_number_of_invariant_closed_geodesics_ACTA}, which are valid for isometries of finite order. For any fixed $m\in\N$, consider a positive rational number $\tau$, and two positive integers $q,r\in\N$ such that
\begin{itemize}
\item $(mp+1)/\tau\in\N$,
\item $q$ is a multiple of $s$,
\item $I^r(\gamma(t))=\gamma(t+\tau)$,
\item $(mp+1)r/\tau\equiv 1$ mod $q$.
\end{itemize}
These conditions imply
\begin{gather*}
\gamma\in\Lambda^\tau(\fix(I^q);I^r)\subset\Lambda^{mp+1}(M;I)\cap\Lambda^q(M;\id),\\
E^{mp+1}|_{\Lambda^\tau(\fix(I^q);I^r)}=E^{q}|_{\Lambda^\tau(\fix(I^q);I^r)}=E^{\tau}|_{\Lambda^\tau(\fix(I^q);I^r)}. 
\end{gather*}
By \cite[Lemma~2.9]{Grove_Tanaka:On_the_number_of_invariant_closed_geodesics_ACTA}, there exist positive rational numbers $\tau_1,...,\tau_n,q_1,...,q_n$, positive integers $r_1,...,r_n$, and a partition $\M_1\cup...\cup\M_n$ of the set of positive integers $\N$ such that, for all $i=1,...,n$ and $m\in\M_i$, the four points above are verified by $\tau:=\tau_i$, $q:=q_i$, and $r:=r_i$, and we have
\begin{align*}
\nul(E^{mp+1}|_{\Lambda^{mp+1}(\fix(I^s);I)},\gamma)
=
\nul(E^{mp+1}|_{\Lambda^{\tau}(\fix(I^q);I^r)},\gamma),\qquad\forall m\in\M_i. 
\end{align*}
Another application of the argument in the proof of~\cite[Proposition~3.5]{Grove_Tanaka:On_the_number_of_invariant_closed_geodesics_ACTA} implies that the gradient $\nabla E^{mp+1}(\zeta)$ is tangent to the submanifold $\Lambda^{\tau}(\fix(I^q);I^r)$ for all $\zeta\in\Lambda^{\tau}(\fix(I^q);I^r)$. If we set 
\[\Omega_i:=\Lambda^{\tau_i}(\fix(I^{q_i});I^{r_i}),\] 
the arguments given so far prove points (i), (ii), and (iii) of the lemma. Notice that, for each $i=1,...,n$, there are finitely many positive rational numbers $q_1'',...,q_u''\in[0,q_i)$ such that, for each $m\in\M_i$, $mp+1$ is congruent to an element of $\{q_1'',...,q_u''\}$ modulo $q_i$. Therefore, point~(iv) follows by taking a suitable subpartition of $\M_1\cup...\cup\M_n$.

Assume now that the period $p$ is irrational. We follow the arguments in \cite[Section~3]{Tanaka:On_the_existence_of_infinitely_many_isometry_invariant_geodesics}. For all $m\in\N$ and for all multiples $q$ of the period $p$, the intersection 
\begin{align*}
\Lambda^{m,q}:=\Lambda^{mp+1}(M;I)\cap\Lambda^{q}(M;\id) 
\end{align*}
is a totally geodesic submanifold of $\Lambda^{mp+1}(M;I)$, and a standard computation shows that the gradient $\nabla E^{mp+1}(\zeta)$ is tangent to $\Lambda^{m,q}$ for all $\zeta\in\Lambda^{m,q}$. Fix one such $\zeta$. For each $\epsilon>0$ there exists $\delta\in[0,\epsilon]$ such that
\begin{align*}
\int_0^\delta g_{\zeta(t)}(\dot\zeta(t),\dot\zeta(t))\,\diff t<\epsilon.
\end{align*}
Since $p$ is irrational, the quotient $(mp+1)/q$ is irrational as well. Therefore, there exist $a,b\in\N$ such that $0\leq a(mp+1)-bq\leq\delta$. We infer that
\[a(mp+1) E^{mp+1}(\zeta) - bq E^q(\zeta)\in[0,\epsilon],\] and therefore
\begin{gather*}
E^{mp+1}(\zeta)
\leq
\frac{a(mp+1)}{bq} E^{mp+1}(\zeta)
\leq E^q(\zeta) + \frac{\epsilon}{bq},\\
E^q(\zeta)
\leq 
\frac{a(mp+1)}{bq} E^{mp+1}(\zeta)
\leq
\left(1+\frac{\epsilon}{bq}\right) E^{mp+1}(\zeta).
\end{gather*}
Since $\epsilon>0$ can be taken arbitrarily small, we have proved that 
\begin{align*}
E^{mp+1}|_{\Lambda^{m,q}}
=
E^{q}|_{\Lambda^{m,q}}
\end{align*}
By \cite[Lemma~3.2]{Tanaka:On_the_existence_of_infinitely_many_isometry_invariant_geodesics}, we infer that there exist positive real numbers $q_1,...,q_n$ that are multiples of $p$, and a partition $\M_1\cup...\cup\M_n$ of the set $\N\setminus\{1,2,...,n_0\}$, for some $n_0\in\N$, such that
\begin{align*}
\nul(E^{mp+1},\gamma)=\nul(E^{mp+1}|_{\Lambda^{m,q_i}},\gamma),
\qquad
\forall m\in\M_i.
\end{align*}
Notice that $\Lambda^{m,q_i}=\Lambda^{m',q_i}$ if $m\equiv n$ mod $q_i/p$. Therefore, up to replacing the partition $\M_1\cup...\cup\M_n$ with a subpartition, we can assume that $\Omega_i:=\Lambda^{m,q_i}=\Lambda^{m',q_i}$ for all $i=1,...,n$ and $m,m'\in\M_i$. This proves points (i), (ii), and (iii) of the lemma in the case where $p$ is irrational. As before, point (iv) follows if we take a further suitable subpartition of $\M_1\cup...\cup\M_n$.
\end{proof}

Let us equip $\Omega_i$ with the complete Riemannian metric pulled back from the one of the space $\Lambda^{mp+1}(M;I)$ via the inclusion, for an arbitrary $m\in\M_i$. The restricted energy function $E^{q_i}|_{\Omega_i}$ satisfies the Palais-Smale condition. Indeed, $E^{q_i}|_{\Omega_i}=E^{mp+1}|_{\Omega_i}$ and the gradient $\nabla E^{mp+1}$ is tangent to  $\Omega_i$; this implies that all Palais-Smale sequences for $E^{q_i}|_{\Omega_i}$ are Palais-Smale sequences for $E^{mp+1}$ as well, and therefore are compact.

For a given $m\in\M_i$, let $h\in\N$ be such that $mp+1=hq_i+q_i'$. Points (i) and (iv) of Lemma~\ref{l:Grove_Tanaka} imply that, for all $\zeta\in\Omega_i$, we have that
\[
I(\zeta(t))=\zeta(t+mp+1)=\zeta(t+q_i'),\qquad\forall t\in\R,
\]
and
\[
\int_0^{mp+1} g_{\zeta(t)}(\dot\zeta(t),\dot\zeta(t))\,\diff t
=
h\int_0^{q_i} g_{\zeta(t)}(\dot\zeta(t),\dot\zeta(t))\,\diff t
+
\int_0^{q_i'} g_{\zeta(t)}(\dot\zeta(t),\dot\zeta(t))\,\diff t.
\]
If $q_i'>0$, this can be rephrased by saying that $\Omega_i$ is contained in $\Lambda^{q_i'}(M;I)$ and \[(mp+1) E^{mp+1}|_{\Omega_i}=hq_i E^{q_i}|_{\Omega_i}+ q_i' E^{q_i'}|_{\Omega_i},\] which, together with Lemma~\ref{l:Grove_Tanaka}(ii), implies that
\begin{align*}
 E^{mp+1}|_{\Omega_i}= E^{q_i}|_{\Omega_i}= E^{q_i'}|_{\Omega_i},\qquad\forall m\in\M_i.
\end{align*}

\section{Broken geodesics}\label{s:broken_geodesics}

\subsection{A space of broken  geodesics}
In Section~\ref{s:main_lemma}, we will make use of the well-known property that closed sublevel sets of the energy function are homotopy equivalent to a compact manifold with boundary. This compact manifold is a subset of Morse's space of broken  geodesics, which we now recall. We refer the reader to~\cite[chap.~4]{Mazzucchelli:Critical_point_theory_for_Lagrangian_systems} and to the references therein for more background details. 

Consider a vector $\ttau=(\tau_0,\tau_1,...,\tau_k)\in\R^{k+1}$ such that $0=\tau_0<\tau_1<...<\tau_k$. We denote by $\Upsilon'\subset C([0,\tau_k];M)$ the subspace of all continuous curves $\zeta:[0,\tau_k]\to M$ such that $\zeta(0)=\zeta(\tau_k)$ and, for all $i=0,...,k-1$, $\zeta|_{[\tau_i,\tau_i+1]}$ is a geodesic of length less than the injectivity radius of $(M,g)$. The space $\Upsilon'$ is diffeomorphic to an open neighborhood of the diagonal in the $k$-fold product $M\times...\times M$ via the map $\zeta\mapsto(\zeta(\tau_1),\zeta(\tau_2),...,\zeta(\tau_{k}))$. As such, it inherits the Riemannian metric $g\oplus...\oplus g$. A curve $s\mapsto \zeta_s$ is a geodesic of $(\Upsilon',g\oplus...\oplus g)$ if and only if, for all $i=0,...,k-1$, the curve $s\mapsto\zeta_s(\tau_i)$ is a geodesic of $(M,g)$.

For some $k'\in\{0,...,k-1\}$, we introduce the subspace $\Upsilon\subset\Upsilon'$ of all $\zeta$'s such that $I(\zeta(0))=\zeta(\tau_{k'})$. It is easy to see that $\Upsilon$ is an embedded smooth submanifold of $\Upsilon'$. Indeed, if $k'\neq0$, we have that 
$\Upsilon=F^{-1}(\mathrm{graph}(I))$, where  $F:\Upsilon'\to M\times M$ is the smooth submersion given by $F(\zeta)=(\zeta(0),\zeta(\tau_{k'}))$. If $k'=0$, we have that $\Upsilon=G^{-1}(\fix(I))$, where $G:\Upsilon'\to M$ is the smooth submersion given by $G(\zeta)=\zeta(0)$, and $\fix(I)$ if the fixed point set of $I$, which is a disjoint union of finitely many closed totally geodesic submanifolds of $M$, see \cite[page~59]{Kobayashi:Transformation_groups_in_differential_geometry}. In both cases, $\Upsilon$ itself is a possibly disconnected, complete, and totally geodesic submanifold of $\Upsilon'$. We denote by $\dist_M:M\times M\to[0,\infty)$ the Riemannian distance on $M$, and we define the function $\dist_{\Upsilon}:\Upsilon\times\Upsilon\to$ by
\[
\dist_\Upsilon(\zeta_0,\zeta_1):=\max_{i=0,...,k-1} \dist_M(\zeta_0(\tau_i),\zeta_1(\tau_i)).
\]
This function is a distance on the connected components of $\Upsilon$  equivalent to the one induced by the Riemannian metric $g\oplus...\oplus g$.

\subsection{Two energy functions}
On $\Upsilon$ we will need to consider two energy functions $F^{\tau_{k}}:\Upsilon\to[0,\infty)$ and $F^{\tau_{k'}}:\Upsilon\to[0,\infty)$. The first one is defined by
\begin{align*}
F^{\tau_k}(\zeta):=
\frac{1}{\tau_k} \int_0^{\tau_k} g_{\zeta(t)}(\dot\zeta(t),\dot\zeta(t))\,\diff t=
\frac{1}{\tau_{k}} \sum_{i=0}^{k-1} 
\frac{\dist_M(\zeta(\tau_i),\zeta(\tau_{i+1}))^2}{\tau_{i+1}-\tau_i}.
\end{align*}
If $k'\neq0$, that is, if $\tau_{k'}>0$, the second one is defined analogously by
\begin{align*}
F^{\tau_{k'}}(\zeta):=
\frac{1}{\tau_{k'}} \int_0^{\tau_{k'}} g_{\zeta(t)}(\dot\zeta(t),\dot\zeta(t))\,\diff t=
\frac{1}{\tau_{k'}} \sum_{i=0}^{k'-1} 
\frac{\dist_M(\zeta(\tau_i),\zeta(\tau_{i+1}))^2}{\tau_{i+1}-\tau_i}.
\end{align*}
If $\tau_{k'}=0$, this second energy function will not be relevant, and we simply set it to be  $F^{\tau_{k'}}\equiv0$. Let us fix, once for all, an energy bound $b>0$. Notice that 
\[\dist_M(\zeta(\tau_i),\zeta(\tau_{i+1}))^2\leq(\tau_{i+i}-\tau_i)\tau_k F^{\tau_k}(\zeta).\] 
Therefore, if we choose $\ttau$ such that
\begin{align}\label{e:delta_tau}
\max\{\tau_{i+1}-\tau_i\ |\ i=0,...,k-1\}<\injrad(M,g)^2/(9 \tau_k b), 
\end{align}
we have
\begin{align*}
 \dist_M(\zeta(\tau_i),\zeta(\tau_{i+1}))
 <
 \injrad(M,g)/3,\qquad\forall \zeta\in\{F^{\tau_k}\leq b\},\ i=0,...,k-1.
\end{align*}
This implies that the sublevel set $\{F^{\tau_k}\leq b\}$ is compact. Actually, all sublevel sets $\{F^{\tau_k}\leq b'\}$ are compact provided $b'<9b$. On the other hand, the sublevel set $\{F^{\tau_{k'}}\leq b\}$ is never compact, because the energy function $F^{\tau_{k'}}$ is independent of $\zeta(\tau_{i})$ for all $i>k'$. Nevertheless, with the above choice of $\ttau$, the intersection $\{F^{\tau_k}\leq b\}\cap\{F^{\tau_{k'}}\leq b\}$ is compact.

\subsection{Polydiscs of broken geodesics}\label{s:polydiscs}
We fix a radius $R\in(0,\injrad(M,g)/3)$ small enough so that the Riemannian balls of $M$ of radius less than or equal to $R$ are geodesically convex: any pair of points in one such ball are joined by a unique shortest geodesic that is entirely contained in the ball (see e.g.\ \cite[page~76]{do_Carmo:Riemannian_geometry}). We will be dealing with polydisc neighborhoods of elements $\gamma\in\{ F^{\tau_k}\leq b\}$ of the form
\[
W(\gamma,r):=\big\{ \zeta\in\Upsilon \ \big|\ \dist_M(\zeta(\tau_i),\gamma(\tau_i))<r\quad \forall i=1,...,k \big\},
\]
where $r\in(0,R]$. This polydisc is simply the ball with radius $r$ centered at $\gamma$ for the distance $\dist_{\Upsilon}$. Since the sublevel set $\{F^{\tau_k}\leq b\}$ is compact, the restriction of the  function $F^{\tau_k}$ to it is uniformly continuous. Hence, we can require the upper bound $R$ for the radii to be small enough (depending on $b$) so that the closure of each of the above polydiscs is compact in $\Upsilon$.

A useful property of  $W(\gamma,r)$ is that every pair of  points in it can be joined by a unique shortest geodesic of $(\Upsilon,g\oplus...\oplus g)$. More specifically, for each $x,y\in M$ with $\dist_M(x,y)<\mathrm{injrad}(M,g)$, we denote by 
\begin{align}\label{e:geodesic_joining_x_y}
 \gamma_{x,y}:[0,1]\to M
\end{align}
the unique shortest geodesic of $(M,g)$ such that $\gamma(0)=x$ and $\gamma(1)=y$. It is well known that $\gamma_{x,y}$  depends smoothly on $(x,y)$, see e.g.\ \cite[Thm.~4.1.2]{Mazzucchelli:Critical_point_theory_for_Lagrangian_systems}. For each pair of elements $\zeta_0,\zeta_1\in W(\gamma,r)$, the unique shortest  geodesic $s\mapsto\zeta_s$  joining them is given by 
\[\zeta_s(\tau_i)=\gamma_{\zeta_0(\tau_i),\zeta_1(\tau_i)}(s),\qquad\forall i=0,...,k-1.\]
Indeed, by the geodesic convexity of Riemannian balls of radius $r$, for all $i\in\Z_k$ and $s\in[0,1]$ we have $\dist_M(\zeta_s(\tau_i),\gamma(\tau_i))< r$. Moreover 
\begin{align*}
\dist_M(\zeta_s(\tau_i),\zeta_s(\tau_{i+1})) 
 <\, & \dist_M(\zeta_s(\tau_i),\gamma(\tau_i))\\ 
 & + \dist_M(\gamma(\tau_i),\gamma(\tau_{i+1}))\\
 & +\dist_M(\gamma(\tau_{i+1}),\zeta_s(\tau_{i+1})) \\
<\, & r + \mathrm{injrad}(M,g)/3 +r \\
 <\, &\mathrm{injrad}(M,g). 
\end{align*}
Therefore, the whole curve $s\mapsto\zeta_s$ is well defined and lies inside $W(\gamma,r)$. In other words, $W(\gamma,r)$ is geodesically convex.

\subsection{Deformation to the space of broken geodesics}\label{s:deformation_to_Upsilon} 
Now, let $\gamma\in\Lambda(M;I)$ be an $I$-invariant geodesic that is periodic with minimal period $p\geq 1$. Consider the objects given by Grove and Tanaka's Lemma~\ref{l:Grove_Tanaka}: a period $q:=q_i$ that is a multiple of $p$, a set of integers $\M=\M_i\subset \N$, a non-negative number $q'=q'_i\in[0,q)$ such that $mp+1\equiv q'$ mod $q$, and a Hilbert manifold $\Omega=\Omega_i\subset \Lambda^{q}(M;\id)\cap\Lambda^{mp+1}(M;I)$. We will choose our energy bound $b$ of the previous section larger than $E(\gamma)$, and the vector $\ttau$ that enters the definition of $\Upsilon$ such that $\tau_{k'}=q'$ and $\tau_k=q$. 

For all $m\in\M$, we will see the space $\Upsilon$ as a submanifold of the Hilbert manifold $\Lambda^{mp+1}(M;I)$ via the embedding 
\begin{align*}
\omega^m:\Upsilon\hookrightarrow\Lambda^{mp+1}(M;I)
\end{align*}
uniquely defined by  
\begin{align*}
\omega^m(\zeta)(hq+t)=\zeta(t),\qquad\forall h\in\N, t\in[0,q] \mbox{ with }hq+t\in[0,mp+1].
\end{align*}
The map $\omega^m$ will play the role of the classical $m$-fold iteration map from the theory of closed geodesics. The various energy functions are related by 
\begin{equation}\label{e:comparison_energies}
\begin{split}
E^{mp+1}\circ\omega^m(\zeta) & = \tfrac{1}{mp+1} \left( \big\lfloor \tfrac{mp+1}{q} \big\rfloor q F^q(\zeta) + q' F^{q'}(\zeta) \right)\\
& = F^q(\zeta) + \tfrac{q'}{mp+1} \big(F^{q'}(\zeta)-F^q(\zeta) \big).
\end{split}
\end{equation}

Now, we  extend the time-intervals given by the vector $\ttau$ in a suitable periodic fashion. We define the bi-infinite sequence  $\{\nu_j\,|\,j\in\Z\}$ as follows. First, we set 
\begin{align*}
\nu_{hk+i}:=hq+\tau_i,\qquad\forall h\in\N,\ i\in\{0,...,k\} \mbox{ with } hq+\tau_i\in[0,mp+1].
\end{align*}
This defines $\nu_j$ for all $j\in\{0,...,j'\}$, where $j'=\lfloor\tfrac{mp+1}{q}\rfloor k+k'$. Notice that $\nu_{j'}=mp+1$. We complete the definition by setting
\begin{align*}
\nu_{hj'+j}=h(mp+1)+\nu_j,\qquad\forall h\in\Z, j\in\{0,...,j'-1\}.
\end{align*}
The sequence $\{\nu_j\,|\,j\in\Z\}$ is precisely the one such that, for all $\zeta\in\Upsilon$ and $i\in\Z$, the restriction  $\omega^m(\zeta)|_{[\nu_i,\nu_{i+1}]}$ is a geodesic of length less than $\injrad(M,g)$.

If $\zeta\in\Omega$ satisfies $E^{mp+1}(\zeta)=E^q(\zeta)\leq b$, then for all time values $t_0,t_1\in\R$ such that $0\leq t_1-t_0<\injrad(M,g)^2/(9qb)$, we have
\begin{align*}
\dist_M(\zeta(t_1),\zeta(t_0))
& \leq
\int_{t_0}^{t_1} \sqrt{g_{\zeta(t)}(\dot\zeta(t),\dot\zeta(t))}\,\diff t\\
& \leq
\sqrt{(t_1-t_0) q b }\\
& <
\injrad(M,g)/3.
\end{align*}
This, together with~\eqref{e:delta_tau}, allows us to define a continuous homotopy
\begin{align}\label{e:broken_deformation_retraction}
 r_s:\{ E^{mp+1}\leq b\}\cap\Omega \to \{ E^{mp+1}\leq b\},\qquad s\in[0,1],
\end{align}
by $r_s(\zeta):=\zeta_s$, where $\zeta_s\in\Lambda^{mp+1}(M;I)$ is the unique curve that coincides with $\zeta$ everywhere, except on the intervals of the form $[\nu_i,(1-s)\nu_{i}+s\nu_{i+1}]$ where it is equal to the shortest geodesic joining its endpoints. Notice that this latter shortest geodesic has length less than $\injrad(M,g)/3$, and therefore is well defined. By construction, this homotopy does not increase the energy. More precisely, for all $i\in\Z$ we have
\begin{align}\label{e:r_t_descreases_energy}
 \frac{\diff}{\diff s} \int_{\nu_i}^{\nu_{i+1}} g_{\zeta_s(t)}(\dot\zeta_s(t),\dot\zeta_s(t))\,\diff t \leq0.
\end{align}
The image of the time-1 map $r_1$ lies inside $\omega^m(\Upsilon)$, i.e.
\begin{align}\label{e:retraction_to_Upsilon}
r_1(\{ E^{mp+1}\leq b\}\cap\Omega)\subset \omega^m(\Upsilon). 
\end{align}
We recall that $E^{mp+1}|_\Omega = E^q|_\Omega$, and if $q'>0$ we further have $E^{mp+1}|_\Omega = E^{q'}|_\Omega$. This, together with~\eqref{e:r_t_descreases_energy}, implies that
\begin{align}\label{e:energy_estimates_for_r1}
\max\{F^{q'}\circ r_1(\zeta),F^q\circ r_1(\zeta)\}\leq E^{mp+1}(\zeta),\qquad \forall\zeta\in\Omega.
\end{align}

\section{The main lemma}\label{s:main_lemma}

\subsection{Statement and outline of the proof}
We say that an $I$-invariant geodesic $\gamma\in\crit(E)$ with minimal period $p\geq1$ is isolated when $\orb(\gamma^{mp+1})$ is isolated in $\crit(E)$ for all $m\in\M$. The following lemma will be the crucial ingredient for our main Theorem~\ref{t:main}. 

\begin{lem}\label{l:iterated_mountain_passes}
Let $\gamma\in\crit(E)$ be an isolated $I$-invariant geodesic  with minimal period $p\geq1$, and fix a degree $d\geq1$. For all $m\in\N$ large enough and for all sufficiently small neighborhoods $V$ of $\orb(\gamma^{mp+1})$ the following holds: given any compact domain $K\subset\R^d$ and any  continuous map $u:K\to \{E<E(\gamma^{mp+1})\}\cup V$ such that $u(\partial K)\cap V=\varnothing$, there exists a homotopy $u_s:K\to \Lambda(M;I)$ such that $u_0=u$, $u_1(K)\subset\{E<E(\gamma^{mp+1})\}$, and $u_s\equiv u$ outside $u^{-1}(V)$ for all $s\in[0,1]$.
\end{lem}

The special case in which the compact domain $K$ is 1-dimensional was established in \cite[Lemma~2.5]{Mazzucchelli:Isometry_invariant_geodesics_and_the_fundamental_group}, and was inspired by an analogous result due to Bangert~\cite{Bangert:Closed_geodesics_on_complete_surfaces} (see also \cite[Theorem~2.6]{Abbondandolo_Macarini_Mazzucchelli_Paternain:Infinitely_many_periodic_orbits_of_exact_magnetic_flows_on_surfaces_for_almost_every_subcritical_energy_level} for a similar result in the context of the Lagrangian free-period action functional). When $K$ has dimension at least 2, the situation becomes much more complicated, and we now give an outline of the argument leading to the proof.

There are two cases to consider. The first, easy one, is when the mean index of the geodesic $\gamma$ is positive, that is to say, the Morse index $\ind(E^{mp+1},\gamma)$ tends to infinity as $m\to\infty$. In particular, if $m$ is large, this index will be larger than the dimension of $K$, and the assertion of the lemma follows from standard arguments of non-linear analysis. This case will be treated in Lemma~\ref{l:iterated_mountain_passes_positive_index}. 

The second, difficult, case, is when the Morse index $\ind(E^{mp+1},\gamma)$ is zero for all $m\in\N$. Here, we need to employ Grove and Tanaka's Lemma~\ref{l:Grove_Tanaka}, which gives us finitely many Hilbert manifolds $\Omega_1,...,\Omega_n$ such that, for all $m\in\N$, the space of $I$-invariant curves $\Lambda^{mp+1}(M;I)$ contains one of the $\Omega_i$'s, and the Morse theory of the function $E^{mp+1}:\Lambda^{mp+1}(M;I)\to\R$ on a neighborhood $V$ of the critical circle $\orb(\gamma)$ is completely described by its restriction to the submanifold $\Omega_i$. The existence of a given map $u$ as in the lemma implies that $\gamma$ is not a global minimum of $E^{mp+1}$ in its connected component of $\Omega_i$. By the broken geodesics approximation of Section~\ref{s:deformation_to_Upsilon}, any sublevel set of the energy on the space $\Omega_i$ can be deformed to a finite dimensional submanifold $\Upsilon$ of broken $I$-invariant geodesics. We introduce a sufficiently fine triangulation $\Sigma$ of the connected component of $\Upsilon$ containing the original $I$-invariant geodesic $\gamma$. By a well-known technique due to Bangert and Klingenberg, for $m$ large enough,  every simplex $\sigma_0\in\Sigma$ can be deformed with a homotopy $\sigma_s$ inside the space $\Lambda^{mp+1}(M;I)$  such that the simplex $\sigma_1$ is contained in the sublevel set $\{E^{mp+1}<E(\gamma)\}$, and all faces of $\sigma_0$ that were already contained in this sublevel set are not moved by the homotopy. Now, we take $m$ large enough so that these homotopies are defined for all the simplexes in $\Sigma$. A map $u$ as in the lemma can be equivalently seen of the form $u:K\to\{E^{mp+1}<E(\gamma)\}\cup V$, where $V\subset\Lambda^{mp+1}(M;I)$ is an open neighborhood of $\gamma$. We take a fine triangulation $\Sigma'$ of the domain $K$. We show that it is possible to deform $u$ inside a subset of $u^{-1}(V)$ such that every simplex $\sigma'\in\Sigma'$ therein is mapped by $u$ to one of the simplexes of $\Sigma$, while the remaining simplexes of $\Sigma'$ are mapped inside the sublevel set $\{E^{mp+1}<E(\gamma)\}$ as they were originally. Finally, Bangert and Klingenberg's homotopies allow to deform $u|_{u^{-1}(V)}$ to a map that takes values inside $\{E^{mp+1}<E(\gamma)\}$. The details of this argument will be carried over in the next four subsections, and will culminate with the proof of Lemma~\ref{l:Bangert_homotopies}, which is Lemma~\ref{l:iterated_mountain_passes} in the case where $\gamma$ has zero mean index.

\subsection{Small neighborhoods of the critical circle}
Throughout Section~\ref{s:main_lemma}, we will work with an isolated $I$-invariant geodesic $\gamma\in\Lambda(M;I)$  with minimal period $p\geq1$ and  zero mean index, meaning that $\ind(E^{mp+1},\gamma)=0$ for all $m\in\N$. The case of positive mean index will be treated  in Section~\ref{s:positive_mean_index}. We set 
\[c:=E^{mp+1}(\gamma)=E(\gamma)=\int_0^1 g_{\gamma(t)}(\dot\gamma(t),\dot\gamma(t))\,\diff t, \]
and we stress that this action is independent of the order of iteration $m\in\N$. Let $\Omega=\Omega_i$ be one of the Hilbert manifolds given by Grove and Tanaka's Lemma~\ref{l:Grove_Tanaka}, $q:=q_i$ the basic period given there, and $\M=\M_i$ the subset of integers $m$ such that claims (i--iv) of the lemma hold. We equip $\Omega$ with the complete Riemannian metric pulled back from the space $\Lambda^{mp+1}(M;I)$ via the inclusion, for an arbitrary $m\in\M$. We denote by $\phi_t$ the anti-gradient flow of the energy $E^q|_\Omega$, which is precisely the restriction to $\Omega$ of the anti-gradient flow of $E^{mp+1}$.

\begin{lem}\label{l:shell}
For every sufficiently small open neighborhood $U'\subset\Omega$ of $\orb(\gamma)$  the following hold.
\begin{itemize}
\item[(i)] For all $\epsilon'>0$ there exists $\tau>0$ such that $\phi_t(U'\cap\{E^q|_\Omega\leq c-\epsilon'\})\cap U'=\varnothing$ for all $t\geq\tau$.
\item[(ii)] There exist $\epsilon>0$  and a smaller open neighborhood $U$ of $\orb(\gamma)$ whose closure is contained in $U'$ such that
$\phi_t(U) \setminus U'\subset \{E^q|_\Omega<c-\epsilon\}$ for all $t\geq0$.
\end{itemize}
\end{lem}

\begin{proof}
We denote by $U(\rho)\subset \Omega$ the open tubular neighborhood of $\orb(\gamma)$ of radius $\rho$. The energy $E^q|_\Omega$ satisfies the Palais-Smale condition. Therefore, for all $0<\rho_1<\rho_2$ such that the closure of $U(\rho_2)$ does not contain critical circles of $E^q|_\Omega$ other than $\orb(\gamma)$, we have
\[
\mu(\rho_1,\rho_2):=\inf\big\{ \|\nabla E^q|_{\Omega}(\zeta)\|\ \big|\ \zeta\in U(\rho_2)\setminus U(\rho_1) \big\}>0.
\]
If $\zeta\in U(\rho_1)$ and $t>0$ are such that $\phi_t(\zeta)\not\in U(\rho_2)$, the curve $\phi_{[0,t]}(\zeta)$ must cross the shell $U(\rho_2)\setminus U(\rho_1)$ of Riemannian width $\rho_2-\rho_1$. Therefore, if we denote by $t'\in(0,t)$  the supremum of the times $s$ such that $\phi_{s}(\zeta)\in U(\rho_1)$, we have
\begin{equation}\label{e:energy_estimate_along_gradient_flow}
\begin{split}
E^q(\zeta)-E^q(\phi_t(\zeta))
&\geq
E^q(\phi_{t'}(\zeta))-E^q(\phi_t(\zeta))\\
&=
\int_{t'}^t \big\|\nabla E^q|_{\Omega}(\phi_s(\zeta))\big\| \cdot \big\|\tfrac{\diff}{\diff s}\phi_s(\zeta)\big\| \,\diff s\\
& \geq (\rho_2-\rho_1)\,\mu(\rho_1,\rho_2).
\end{split}
\end{equation}

We fix once for all two such values $0<\rho_1<\rho_2$, and we require $\rho_2$ to be small enough such that $U(\rho_2)$ does not intersect the sublevel set $\{E^q|_\Omega<c-\epsilon''\}$, for some $\epsilon''>0$. The open neighborhood $U'$ of the lemma must be small enough so that it is contained in $U(\rho_1)$ and  disjoint from the sublevel set $\{E^q|_\Omega\leq c-(\rho_2-\rho_1)\,\mu(\rho_1,\rho_2)\}$. Given $\epsilon'>0$, we set 
\begin{align*}
\nu & :=\inf\big\{ \|\nabla E^q|_{\Omega}(\zeta)\|\ \big|\ \zeta\in U(\rho_2)\cap\{E^q|_{\Omega}\leq c-\epsilon'\} \big\}>0,\\
\tau &:= \epsilon''/\nu^2.
\end{align*}
Let us prove point (i) of the lemma: we claim that, for every  $\zeta\in U'\cap\{E^q|_{\Omega}\leq c-\epsilon'\}$ and   $t\geq\tau$, we have that $\phi_t(\zeta)\not\in U'$. Indeed, the curve $\phi_{[0,t]}(\zeta)$ cannot be entirely contained in $U(\rho_2)$, since otherwise
\begin{align*}
E^q(\phi_t(\zeta)) & = E^q(\zeta) - \int_0^t \big\|\nabla E^q|_{\Omega}(\phi_s(\zeta))\big\|^2\,\diff s\\
&< c- \tau\, \nu^2\\
&= c- \epsilon''.
\end{align*}
Hence, by the energy estimate in~\eqref{e:energy_estimate_along_gradient_flow}, we have
\begin{align*}
E^q(\phi_t(\zeta))\leq E^q(\zeta) - (\rho_2-\rho_1)\,\mu(\rho_1,\rho_2) < c-(\rho_2-\rho_1)\,\mu(\rho_1,\rho_2),
\end{align*}
which proves our claim.

Now, we choose a radius  $r_2>0$ small enough so that $U(r_2)\subset U'$, and we set $\epsilon:=\mu(r_2/2,r_2) r_2/4$. We also choose a radius $r_1\in(0, r_2/2)$ small enough so that 
\[U:=U(r_1)\subset\{E^q|_\Omega<c+\epsilon\}.\] 
Let $\zeta\in U$ be a point such that $\phi_t(\zeta)\not\in U'$ for some $t>0$. In particular, $E^q(\zeta)<c+\epsilon$ and the estimate~\eqref{e:energy_estimate_along_gradient_flow} implies
\begin{align*}
E^q(\zeta)-E^q(\phi_t(\zeta))
&\geq \mu(r_2/2,r_2)\,r_2/2= 2\epsilon.
\end{align*}
This proves point (ii).
\end{proof}

\subsection{Reference simplexes}\label{s:reference_simplexes}
In this section, we will also need to consider the real number $q'=q_i'\in[0,q)$ given by Lemma~\ref{l:Grove_Tanaka}. We fix, once for all, an energy bound $b>c=E^{q}(\gamma)$ and a vector $\ttau=(\tau_0,...,\tau_k)\in\R^{k+1}$ such that $0=\tau_0<\tau_1<...<\tau_k=q$, $\tau_{k'}=q'$ for some $k'\in\{0,...,k-1\}$, and the inequality~\eqref{e:delta_tau} is satisfied. With the notation of Section~\ref{s:broken_geodesics}, we consider the space of broken geodesics $\Upsilon$ associated to $\ttau$, and the energy functions $F^q$ and $F^{q'}$. We will also consider the polydiscs defined in Section~\ref{s:polydiscs}, whose radii are smaller than or equal to $R$.

We fix an arbitrarily small $\epsilon>0$. Since the sublevel set $\{F^q\leq b\}$ is compact, the functions $F^q$ and $F^{q'}$ are uniformly equicontinuous on this set. Namely, there exists $\delta\in(0,R]$ such that
\begin{equation}\label{e:uniform_continuity}
\begin{split}
\max\big\{
F^q(\zeta_0)-F^q(\zeta_1),F^{q'}(\zeta_0)-F^{q'}(\zeta_1)
\big\}
<\epsilon,\\
\forall \zeta_0,\zeta_1\in\{F^q\leq b\}\mbox{ with }\dist_{\Upsilon}(\zeta_0,\zeta_1)<\delta.
\end{split}
\end{equation}
We denote by $\Upsilon''$ the union of those connected components of $\Upsilon$ intersecting the sublevel set $\{F^q<c\}$. We can find a finite subset $\Sigma_0=\{\sigma_1,...,\sigma_h\}\subset \{F^{q}\leq b\}\cap\Upsilon''$ that is $\delta/5$-dense in the sublevel set, that is,
\begin{align}\label{e:reference_0_simplexes}
\{F^{q}\leq b\}\cap\Upsilon'' \subset \bigcup_{\sigma\in\Sigma_0} W(\sigma,\delta/5).
\end{align}
Notice that, since $\delta<R$, the union of the polydiscs $W(\sigma,\delta/5)$ is relatively compact in $\Upsilon$. For each pair of (not necessarily distinct) $\sigma_i,\sigma_j\in\Sigma_0$ such that $\dist_{\Upsilon}(\sigma_i,\sigma_j)<\delta$, we denote by $\sigma_{ij}:[0,1]\to \Upsilon$ the $1$-simplex given by the minimal geodesic of $\Upsilon$ joining $\sigma_i$ and $\sigma_j$. We denote by $\Sigma_1$ the finite collection of all these 1-simplexes. We now inductively define the finite collections $\Sigma_d$ for increasing values of $d$, starting for $d=2$. Given (not necessarily pairwise distinct) $\sigma_{i_0},...,\sigma_{i_d}\in\Sigma_0$ such that $\dist_{\Upsilon}(\sigma_{i_j},\sigma_{i_l})<\delta$ for all $j,l=0,...,d$, we define $\sigma_{i_0...i_d}:\Delta^d\to \Upsilon$ to be the $d$-simplex whose $j$-th face is the $(d-1)$-simplex \[\sigma_{i_0...\widehat{i_{j}}...i_d}\in\Sigma_{d-1},\] and whose interior is defined as follows: we see the standard $d$-simplex 
\[\Delta^d=\big\{x\in[0,1]^d\,\big|\,\textstyle\sum_{i}x_i\leq1\big\}\] as a union of affine curves of slope $(1,...,1)\in\R^d$, so that each $x\in\Delta^d$ belongs to the curve joining the points $\alpha(x),\omega(x)\in\partial\Delta^d$, and the maps $\alpha:\Delta^d\to\partial\Delta^d$ and $\omega:\Delta^d\to\partial\Delta^d$ are continuous; we define the restriction of $\sigma_{i_0...i_d}$ to the curve passing through $x\in\Delta^d$ as the geodesic of $\Upsilon$ joining $\sigma_{i_0...i_d}(\alpha(x))$ and $\sigma_{i_0...i_d}(\omega(x))$. Notice that, by the geodesic convexity of polydiscs of radius less than or equal to  $R$, every $d$-simplex $\sigma_{i_0...i_d}\in\Sigma_d$ is entirely contained in the polydisc of radius $\delta$ centered at any of its $0$-faces.

\subsection{Bangert's homotopies}

For all $m\in\M$, the space of broken geodesics $\Upsilon$ can be seen as a submanifold of $\Lambda^{mp+1}(M;I)$ via the  embedding $\omega^m:\Upsilon\hookrightarrow \Lambda^{mp+1}(M;I)$ defined in Section~\ref{s:deformation_to_Upsilon}. In this section, we will further need to assume that our $I$-invariant geodesic $\gamma$ is not a global minimum of the energy $E^q|_{\Omega}$ inside its connected component of $\Omega$. Since the homotopy $r_s$ of equation~\eqref{e:broken_deformation_retraction} decreases the energy and, in particular, its time-1 map satisfies~\eqref{e:energy_estimates_for_r1}, the element $r_1(\gamma)$ is not a global minimum of $F^q$, nor of $F^{q'}$ if $q'>0$, in its connected component of $\Upsilon$. By modifying an argument due to Bangert and Klingenberg \cite[Thm.~2]{Bangert_Klingenberg:Homology_generated_by_iterated_closed_geodesics} we will show that, if $m$ is large, the images of the simplexes constructed in the previous section under the map $\omega^m$ can be pushed, in a suitable way, below the critical level $c$.

The main ingredient for such statement is Bangert's well known technique of ``pulling one loop at the time''. Given any continuous map 
$\theta_0:\Delta^j\to\Lambda^{q}(M;\id)\subset\Lambda^{mq}(M;\id)$, 
we can deform it with a homotopy $\theta_s:\Delta^j\to\Lambda^{mq}(M;\id)$, $s\in[0,1]$,  such that 
\begin{itemize}
\setlength{\itemsep}{3pt}
\item $\theta_s|_{\partial\Delta^j}=\theta_0|_{\partial\Delta^j}$,
\item $\theta_s(x)(0)=\theta_0(y_s(x))(0)$ for a suitable continuous homotopy $y_s:\Delta^j\to\Delta^j$ such that $y_0=\id$,
\item $E^{mq}\circ\theta_1\leq \max\{E^q(\theta_0(x))\, |\, x\in\partial\Delta^j\}  + \mathrm{const}/m$, where $\mathrm{const}\geq0$ is a quantity depending only on $\theta_0$ (in particular, independent of $m$).
\end{itemize}
The proof of this fact is outlined with our notation and for $j=1$ in \cite[Sect.~3.2]{Mazzucchelli:On_the_multiplicity_of_isometry_invariant_geodesics_on_product_manifolds}, see in particular Figure~1 therein. The case $j>1$ is not harder: it suffices to see the $j$-simplex as a smooth family of $1$-simplexes, and apply to each of them the construction. In the following, we will refer to homotopies of this kind as to Bangert's homotopies.

\begin{lem}\label{l:Bangert_homotopies}
For every sufficiently large integer $m\in\M$ we can associate to each $j$-simplex $\sigma\in\Sigma_0\cup...\cup\Sigma_d$ a homotopy 
\begin{align}\label{e:h_sigma}
h_\sigma:[0,1]\times\Delta^j\to \Lambda^{mp+1}(M;I)
\end{align} 
with the following properties:
\begin{itemize}
\setlength{\itemsep}{3pt}
\item[(i)] $h_\sigma(0,\cdot)=\omega^m\circ\sigma$,
\item[(ii)] if  $\sigma$ is contained in $\{F^q<c\}$, then  $h_\sigma(s,\cdot)=\omega^m\circ\sigma$ for all $s\in[0,1]$,
\item[(iii)] $h_\sigma(s,F_l(\cdot))=h_{\sigma\circ F_l}(s,\cdot)$ for all $l=0,...,j$, where $F_l:\Delta^{j-1}\to\partial\Delta^j$ is the affine map onto the $l$-th face of $\Delta^j$,
\item[(iv)] $h_\sigma(1,\cdot)$ is mapped in $\{E^{mp+1}<c\}$.
\end{itemize}
\end{lem}

\begin{proof}
The proof is essentially the same as the one of \cite[Lemma~3.3]{Mazzucchelli:On_the_multiplicity_of_isometry_invariant_geodesics_on_product_manifolds}. Since there are some technical differences in the setting and in the statement, we provide the complete argument for the reader's convenience.

We define the embedding  
\[\psi:\Upsilon\hookrightarrow\Lambda^q(M;\id)\] 
in the obvious way: for all $\zeta\in\Upsilon$, the image $\psi(\zeta)$ is the unique $q$-periodic curve whose restriction to the interval $[0,q]$ is precisely equal to $\zeta$. Notice that, by the definition of $\Upsilon$, we have that 
\[\psi(\zeta)(mp+1)=\psi(\zeta)(q')=I(\psi(\zeta)(0)),\] 
and indeed $\psi(\zeta)|_{[0,mp+1]}=\omega^m(\zeta)|_{[0,mp+1]}$.

Let us construct the homotopies of the lemma, beginning with the 0-simplexes $\sigma\in\Sigma_0$. Since $F^q(r_1(\gamma))=E^q(\gamma)=c$ and $r_1(\gamma)$ is not a global minimum of $F^q$ in its connected component of $\Upsilon$, there exists a continuous path $s\mapsto \sigma_s\in\Upsilon$ such that $\sigma_0=\sigma$ and $F^q(\sigma_1)<c$. If $F^q(\sigma)<c$ and $F^{q'}(\sigma)<c$ we choose this path to be the stationary one $\sigma_s\equiv\sigma$. We define a map 
\[h_\sigma':[0,1]\times\Delta^0=[0,1]\to\Lambda^q(M;\id)\] 
by
\begin{align*}
h_\sigma'(s)=
\left\{
  \begin{array}{lll}
    \psi(\sigma_{(d+1)s}), &  & \mbox{if }s\in[0,1/(d+1)], \\ 
    \psi(\sigma_1) &  & \mbox{if }s\in[1/(d+1),1]. \\ 
  \end{array}
\right.
\end{align*}
Notice that
\[
E^q(h_\sigma'(s))=F^q(\sigma_1)<c,\qquad\forall s\in\big[ \tfrac{1}{d+1},1 \big].
\]
We also set $m_0:=1$.

We proceed to construct homotopies for the higher dimensional simplexes iteratively in the degree $j$, from $j=1$ upward. We consider a large enough  integer $m_j$ that is divisible by $m_{j-1}$, so that $\Lambda^{m_{j-1}q}(M;\id)\subset\Lambda^{m_{j}q}(M;\id)$. The precise value of $m_j$ will be fixed in a moment. Consider $\sigma\in\Sigma_j$. We construct the maps  
\begin{align*}
h_\sigma'&:[0,1]\times\Delta^j\to\Lambda^{m_jq}(M;\id)\\ 
y_\sigma &:[0,1]\times\Delta^j\to\Delta^j 
\end{align*}
as follows. We set $h_\sigma'(0,\cdot):=\psi\circ\sigma$ and $y_\sigma(0,\cdot):=\id$. For all $l=0,...,j$ and $s\in[0,1]$, we set $h_\sigma'(s,F_l(\cdot)):=h_{\sigma\circ F_l}'(s,\cdot)$ and $y_\sigma(s,F_l(\cdot)):=y_{\sigma\circ F_l}(s,\cdot)$, where the maps $h_{\sigma\circ F_l}'$  and $y_{\sigma\circ F_l}$ were defined in the previous step of the iterative procedure. Up to now, we have defined the maps $h_\sigma'$ and $y_\sigma$ on $(\{0\}\times\Delta^j)\cup([0,1]\times\partial\Delta^j)$. The previous step of this iterative procedure was carried out in such a way that
\begin{align*}
h_\sigma'(s,\cdot)|_{\partial\Delta^j}=h_\sigma'(\tfrac{j}{d+1},\cdot)|_{\partial\Delta^j},\qquad\forall s\in\big[\tfrac{j}{d+1},1\big],
\end{align*}
and
\begin{align}\label{e:bound_previous_step}
\max\big\{E^{m_{j-1}q}(h_\sigma'(\tfrac{j}{d+1},x))\, |\, x\in\partial\Delta^j\big\}<c.
\end{align}
We choose a retraction 
\[\pi_j: [0,j]\times\Delta^j\to(\{0\}\times\Delta^j)\cup([0,j]\times\partial\Delta^j),\] 
and we set 
\begin{align*}
h_\sigma'(s,x):=h_{\sigma}'(s,\pi_j(x)),\qquad y_\sigma(s,x):=y_\sigma(s,\pi_j(x)),\\
\forall[0,j]\times\Delta^j. 
\end{align*}
Finally, we define the homotopies $s\mapsto h_\sigma'(s,\cdot)$ and $s\mapsto y_\sigma(s,\cdot)$, for $s\in\big[\tfrac{j}{d+1},\tfrac{j+1}{d+1}\big]$, by means of a Bangert homotopy associated to the $j$-simplex 
\[h_\sigma'(\tfrac{j}{d+1},\cdot):\Delta^j\to\Lambda^{m_{j-1}q},\] 
and we extend them constantly for $s\in\big[\tfrac{j+1}{d+1},1\big]$. Summing up, we have constructed the maps $h_\sigma'$ and $y_\sigma$ such that
\begin{itemize}
\setlength{\itemsep}{3pt}
\item[(i')] $h_\sigma'(0,\cdot)=\psi\circ\sigma$,
\item[(ii')] $h_\sigma'(s,F_l(\cdot))=h_{\sigma\circ F_l}'(s,\cdot)$ for all $l=0,...,j$,
\item[(iii')] $h_\sigma'(s,\cdot)=h_{\sigma}'(\tfrac{j+1}{d+1},\cdot)$ for all $s\in\big[\tfrac{j+1}{d+1},1\big]$,
\item[(iv')] $E^{m_j q}(h_\sigma'(\tfrac{j+1}{d+1},\cdot))\leq \max\{E^{m_{j-1}q}(h_\sigma'(\tfrac{j}{d+1},x))\, |\, x\in\partial\Delta^j\}  + \tfrac{\mathrm{const}}{m_j/m_{j-1}}$, where $\mathrm{const}\geq0$ is a quantity depending only on $\sigma$,
\item[(v')] $h_\sigma'(s,x)(0)=\sigma(y_\sigma(s,x))(0)$.
\end{itemize}
We choose $m_j$ large enough so that, by (iii'), (iv'), and~\eqref{e:bound_previous_step}, we have
\begin{align}\label{e:estimate_energy_h_sigma'}
E^{m_j q}(h_\sigma'(s,\cdot))< c,\qquad\forall s\in\big[\tfrac{j+1}{d+1},1\big].
\end{align}
If $\sigma$ is mapped inside the sublevel set $\{F^q<c\}$, then we define the above maps simply by $h_\sigma'(s,x):=\sigma(x)$ and $y_\sigma(s,x):=x$, and notice that conditions (i'--v') are still verified.

Now, consider an integer $m\in\M$, that we will require to be large enough in a moment. We set $m':=\big\lfloor\tfrac{mp+1}{m_dq}\big\rfloor$ and $q'':=mp+1-m'm_dq \in[0,m_dq)$, so that
\[mp+1=m'm_dq+q''.\] 
Notice that, by Lemma~\ref{l:Grove_Tanaka}(iv),
\[q''\equiv q'\quad\mathrm{mod}\ q.\] 
For all $\sigma\in\Sigma_0\cup...\cup\Sigma_d$, we define the homotopy $h_\sigma$ of the lemma by
\begin{align*}
h_\sigma(s,x)(t):=
\left\{
  \begin{array}{lll}
    h_\sigma'(s,x)(t), &  & \mbox{if }t\in\left[0, m'm_dq \right], \\ 
    \psi\circ\sigma(y_\sigma(s,x))(t-m'm_dq) &  &  \mbox{if }t\in\left[m'm_dq,mp+1 \right].
  \end{array}
\right.
\end{align*}
Notice that
\begin{align*}
h_\sigma(s,x)(mp+1)&=\psi\circ\sigma(y_\sigma(s,x))(q'')\\
&=\sigma(y_\sigma(s,x))(q')\\
&=I(\sigma(y_\sigma(s,x))(0))\\
&= I(h_\sigma'(s,x)(0))\\
&= I(h_\sigma(s,x)(0)).
\end{align*}
Therefore, $h_\sigma$ is a well defined map of the form~\eqref{e:h_sigma}. Properties~(i--iii) readily follow from its construction. As for property~(iv),  we have
\begin{align*}
E^{mp+1}(h_\sigma(1,x))
& =
\frac{1}{mp+1}
\Big(
m'm_dq E^{m_dq}(h_\sigma'(1,x))\\ 
&\qquad\qquad\quad + \lfloor q''/q\rfloor q F^q(\sigma(y_\sigma(1,x)))\\
&\qquad\qquad\quad + q' F^{q'}(\sigma(y_\sigma(1,x)))
\Big)\\
& = E^{m_dq}(h_\sigma'(1,x)) + \frac{c'}{mp+1}, 
\end{align*}
where 
\begin{align*}
c'  :=\, &\lfloor q''/q\rfloor q F^q(\sigma(y_\sigma(1,x)))+ q' F^{q'}(\sigma(y_\sigma(1,x)))-q''E^{m_dq}(h_\sigma'(1,x))\\
<\, & m_d q \max \{F^q\circ\sigma\}   + q' \max \{F^{q'}\circ\sigma\}.
\end{align*}
Notice that the quantity $c'$ is uniformly bounded independently of $m\in\M$. Since $E^{m_dq}(h_\sigma'(s,x))<c$, if $m\in\M$ is large enough we have $E^{mp+1}(h_\sigma(1,x))<c$, which proves property~(iv).
\end{proof}

\subsection{The case of zero mean index}\label{s:main_lemma_zero_index}
For every integer $m\in\M$, Lemma~\ref{l:Grove_Tanaka}(iii) implies that $\nul(E^q|_\Omega,\gamma)=\nul(E^{mp+1},\gamma)$, and, by our assumption, $\ind(E^q|_\Omega,\gamma)=\ind(E^{mp+1},\gamma)=0$. Let $U'\subset\Omega$ be a small enough open neighborhood of $\orb(\gamma)$ such that Lemma~\ref{l:shell} applies to it, and, for some $h>0$ and for all $\zeta\in U'$ and vector fields $\xi\in\Tan_{\zeta}\Lambda^{mp+1}(M;I)$ orthogonal to $\Omega$, we have
\begin{align}\label{e:positive_definite}
\Hess E^{mp+1}(\zeta)[\xi,\xi]\geq h \|\xi\|_\zeta^2.
\end{align}
Here, the norm on the right hand side is the one induced by the Riemannian metric on the Hilbert manifold $\Lambda^{mp+1}(M;I)$.   Lemma~\ref{l:shell} gives us a quantity $\epsilon>0$ and a smaller open set $U$ of $\orb(\gamma)$ whose closure is contained in $U'$ satisfying the properties stated there. Within the proof of Lemma~\ref{l:iterated_mountain_passes_zero_mean_index}, we will also need to fix $\delta>0$ small enough so that~\eqref{e:uniform_continuity} holds, and we will consider the $\delta/5$-dense reference simplexes built in Section~\ref{s:reference_simplexes}.

Let $U_m\subset\Lambda^{mp+1}(M;I)$ be a small tubular neighborhood of $U$, which is diffeomorphic to an open neighborhood of the zero-section in the normal bundle of $U$. Consider the associated radial deformation retraction 
\begin{align}\label{e:Morse_deformation_retraction}
 f_s:U_m\to U_m,
\end{align}
which satisfies $f_0=\id$ and $f_1(U_m)=U$. Since, by Lemma~\ref{l:Grove_Tanaka}(i), the gradient of $\nabla E^{mp+1}$ is tangent to $\Omega$ along $\Omega$, the restriction of the energy $E^{mp+1}$ to a fiber of $f_1$ has a critical point at the intersection of the fiber with $\Omega$. By the convexity~\eqref{e:positive_definite}, this critical point is a strict minimum. In particular, if the tubular neighborhood $U_m$ is sufficiently small, the deformation $f_s$ does not increase the energy, i.e.\ 
$\tfrac{\diff}{\diff s} E^{mp+1}\circ f_s\leq0$.

\begin{lem}\label{l:iterated_mountain_passes_zero_mean_index}
Let $\gamma\in\crit(E)$ be an isolated $I$-invariant geodesic that is periodic with minimal period $p\geq1$, having energy $c:=E(\gamma)$, and satisfying $\ind(E^{mp+1},\gamma)=0$ for all $m\in\N$. Fix a degree $d\geq1$. For all $m\in\N$ large enough and for all sufficiently small neighborhoods $V\subset\Lambda^{mp+1}(M;I)$ of $\orb(\gamma)$ the following holds: given any compact domain $K\subset\R^d$ and any continuous map $u:K\to \{E^{mp+1}<c\}\cup V$ such that $u(\partial K)\cap V=\varnothing$, there exists a homotopy $u_s:K\to \Lambda^{mp+1}(M;I)$ such that $u_0=u$, $u_1(K)\subset \{E^{mp+1}<c\}$, and $u_s\equiv u$ outside $u^{-1}(V)$  for all $s\in[0,1]$.
\end{lem}

\begin{proof}
Consider one of the infinite subsets $\M$ of the partition $\M_1\cup...\cup\M_n$  given by Lemma~\ref{l:Grove_Tanaka}. We prove the lemma for any $m\in\M$ large enough so that the homotopies of Lemma~\ref{l:Bangert_homotopies} exist.

The neighborhood $V\subset\Lambda^{mp+1}(M;I)$ of the critical circle $\orb(\gamma)$ must be small enough so that it is contained into the neighborhood $U_m$ constructed above. Let $V''\Subset V'\Subset V$ be  slightly smaller  neighborhoods of $\orb(\gamma)$ in $\Lambda^{mp+1}(M;I)$ such that 
$u(K)\setminus V''\subset\{E^{mp+1}<c\}$. Here, the symbol $\Subset$ means that the closure of the former set in contained in the latter set. Fix $\epsilon'>0$ small enough so that 
\begin{align}\label{e:outside_below}
u(K)\setminus V''\subset\{E^{mp+1}\leq c-\epsilon'\}.
\end{align} 
We set $A'':=u^{-1}(V'')$, $A':=u^{-1}(V')$ and $A:=u^{-1}(V)$. Since $u(\partial K)\cap V=\varnothing$, we have that 
\[A''\Subset A'\Subset A.\] 
Moreover, \eqref{e:outside_below} can be rewritten as 
$u(K\setminus A'')\subset\{E^{mp+1}\leq c-\epsilon'\}$. We introduce two  bump functions: a smooth function  $\chi':K\to[0,1]$ that is identically equal to $1$ on $A'$ and has support inside $A$, and  a smooth function  $\chi'':K\to[0,1]$ that is identically equal to $1$ on $A''$ and has support inside $A'$.

Consider the deformation  retraction  $f_s$ of equation~\eqref{e:Morse_deformation_retraction}. We replace our $u$ of the statement with $x\mapsto f_{\chi'(x)}\circ u(x)$, so that 
$u(\overline{A'})\subset U\subset\Omega$. We denote by $\phi_s$ the anti-gradient flow of $E^{mp+1}$, and we recall that, by Lemma~\ref{l:Grove_Tanaka}(i), $\Omega$ is invariant by this flow. By~\eqref{e:outside_below} we have that
\[u(\partial A'')\subset \{E^{mp+1}|_\Omega < c-\epsilon'\}.\] 
By Lemma~\ref{l:shell}, there exists $\tau>0$ large enough so that $\phi_\tau( u(\partial A''))\cap U'=\varnothing$ and 
\[\phi_{\tau}(u(A''))\setminus U'\subset\{E^{mp+1}<c-\epsilon\}.\]
We replace $u$ with $x\mapsto \phi_{\tau\chi''(x)}\circ u(x)$.

Summing up, we have deformed $u$ by means of a homotopy supported inside $A$ to a new map, which we still denote by $u$, satisfying
\begin{align}
\nonumber u(K\setminus A'') &\subset \{E^{mp+1}<c\},\\
\nonumber u(A') &\subset \Omega,\\
\nonumber u(\partial A'')\cap U' &=\varnothing,\\
u(A'')\setminus U' & \subset \{E^{mp+1}<c-\epsilon\}. \label{e:u_outside_U}
\end{align}
Now, we make a partition 
\[
A''=C\cup D
\] 
in two regions: the open shell $C:=A''\setminus \overline{u^{-1}(U')}$, and the compact subset $D:= A''\setminus C$. Consider the homotopy $r_s$ of equation~\eqref{e:broken_deformation_retraction}. We further replace $u$ with 
$x\mapsto r_{\chi''(x)}\circ u(x)$, so that, by~\eqref{e:retraction_to_Upsilon}, we have that $u(A'')\subset \omega^m(\Upsilon)$. Since $\omega^m$ is an embedding, there exists a unique continuous map $v:A''\to\Upsilon$   such that 
\[u|_{A''}=\omega^m\circ v.\]
Equations~\eqref{e:u_outside_U} and~\eqref{e:energy_estimates_for_r1} imply
\begin{align}\label{e:v(C)_contained_low}
v(C)\subset\{F^q<c-\epsilon\}\cap\{F^{q'}<c-\epsilon\}.
\end{align}

Now, consider the simplexes in $\Sigma_1,...,\Sigma_d$ constructed in Section~\ref{s:reference_simplexes}. We introduce a sufficiently fine triangulation of a compact neighborhood $\Sigma'\subset A''$ of $D$, so that the image of every simplex under $v$ is contained in a polydisc of $\Upsilon$ of diameter $\delta/5$ (see Section~\ref{s:polydiscs}). Up to further reducing the size of all simplexes by applying finitely many  barycentric subdivisions to $\Sigma'$, we can assume that the subset $\Sigma''$ given by all the simplexes intersecting $D$ is contained in the interior of $\Sigma'$. We denote by $\Sigma_j'$, for $j=0,...,d$, the collection of $j$-dimensional faces of the simplexes in $\Sigma'$. We choose a map $\iota_0:\Sigma_0'\to\Sigma_0$ such that $\dist(v(\sigma),\iota_0(\sigma))<\delta/5$ for all $\sigma\in\Sigma_0'$. The existence of such a map is guaranteed by equation~\eqref{e:reference_0_simplexes}. We  define maps $\iota_j:\Sigma_j'\to\Sigma_j$ in a compatible way: given a $j$-simplex $\sigma\in\Sigma_j'$ whose ordered set of $0$-faces is $\sigma_0,...,\sigma_j\in\Sigma_0'$, its image $\iota_j(\sigma)$ is the (unique) simplex in $\Sigma_j$ whose ordered set of $0$-faces is $\iota_0(\sigma_0),...,\iota_0(\sigma_j)\in\Sigma_0$. Notice that such a $\iota_j(\sigma)$ exists, since
\begin{align*}
 \dist_{\Upsilon}(\iota_0(\sigma_h),\iota_0(\sigma_l))
 \leq\, &
 \dist_{\Upsilon}(\iota_0(\sigma_h),v(\sigma_h))\\
 & +
 \dist_{\Upsilon}(v(\sigma_h),v(\sigma_l)) \\ 
 & +
 \dist_{\Upsilon}(v(\sigma_l),\iota_0(\sigma_l)) \\
 <\, & \tfrac15 \delta+\tfrac15 \delta+\tfrac15 \delta\\
=\, & \tfrac35\delta.
\end{align*}
In particular, the $0$-faces $\iota_0(\sigma_1),...,\iota_0(\sigma_j)$ are contained in  $W(\iota_0(\sigma_0),\tfrac35\delta)$. Since this polydisc is geodesically convex, the whole simplex $\iota_j(\sigma)$  is contained in $W(\iota_0(\sigma_0),\tfrac35\delta)$. The $C^0$-distance between the simplexes $v\circ\sigma$ and $\iota_j(\sigma)$ can be estimated as
\begin{equation}\label{e:C0_distance_from_references}
 \begin{split}
 \dist_{\Upsilon}(v(\sigma(x)),\iota_j(\sigma)(x))
 \leq\, &
 \dist_{\Upsilon}(v(\sigma(x)),v(\sigma_0))\\
& +
 \dist_{\Upsilon}(v(\sigma_0),\iota_0(\sigma_0))\\
& +
 \dist_{\Upsilon}(\iota_0(\sigma_0),\iota_j(\sigma)(x)) \\
 <\, & \tfrac15 \delta +\tfrac15 \delta +\tfrac35 \delta\\
 =\, &\delta.
\end{split}
\end{equation}
We define a homotopy $h_s:\Sigma'\to \Upsilon$ simplex by simplex as follows. Consider any $d$-simplex  $\sigma:\Delta^d\to \Sigma'$  belonging to $\Sigma_d'$, and any point $x\in\Delta^d$. We set $s\mapsto h_s(\sigma(x))$ to be the unique minimal geodesic of $\Upsilon$ joining $h_0(\sigma(x))=v(\sigma(x))$ and $h_1(\sigma(x))=\iota_d(\sigma)(x)$. Let $\chi:A''\to[0,1]$ be a smooth bump function that is identically 1 on $\Sigma''$ and is supported inside $\Sigma'$. We define a homotopy $v_s:A''\to \Upsilon$ by
\[
v_s(x):=h_{s\chi(x)}(x),
\]
so that $v_0=v$, $v_s|_{\Sigma''}=h_s|_{\Sigma''}$, and $v_s$ is identically equal to $v$ outside $\Sigma'$.  Notice that, by the very construction of it, the homotopy $v_s$ still satisfies the same estimate as in~\eqref{e:C0_distance_from_references}, i.e.
\begin{equation*}
 \dist_{\Upsilon}(v_s(\sigma(x)),\iota_j(\sigma)(x))<\delta.
\end{equation*}
This, together with~\eqref{e:v(C)_contained_low} and the uniform equicontinuity~\eqref{e:uniform_continuity}, implies that
\begin{align}\label{e:vs_outside_Sigma''}
 v_s(A''\setminus\Sigma'')\subset v_s(C)\subset \{F^q<c\}\cap\{F^{q'}<c\}.
\end{align}
By equation~\eqref{e:comparison_energies}, we have
\begin{align*}
\omega^m\circ v_s(C)\subset \{E^{mp+1}<c\}.
\end{align*}

Now, we apply Lemma~\ref{l:Bangert_homotopies}: for every $\sigma\in\Sigma'_d$ we obtain a Bangert homotopy $h_{v_1(\sigma)}$ with the properties (i--iv) given there. Notice that, by property~(iii), these homotopies coincide on common faces of simplexes in $\Sigma'$. Moreover, by~\eqref{e:vs_outside_Sigma''} and property~(ii), for every $\sigma\in\Sigma'_d\setminus\Sigma''_d$ the homotopy $h_{v_1(\sigma)}$ is identically equal to $\omega^m\circ v_1\circ\sigma$. By patching together  the homotopies $h_{v_1(\sigma)}$,  we can build a homotopy 
\[w_s:A''\to \Lambda^{mp+1}(M;I),\qquad s\in[0,1],\] so that $w_0=\omega^m\circ v_1$, $w_s$ is identically equal to $\omega^m\circ v_1$ outside  $\Sigma''$, and \[w_1(A'')\subset\{E^{mp+1}<c\}.\] Finally, we construct the homotopy $u_s$ of the lemma as
\[
u_s(x)
:=
\left\{
  \begin{array}{lll}
    u(x), &  &\mbox{if } x\in K\setminus A'', \\ 
    \omega^m\circ v_{2s}(x),\Big. &  & \mbox{if } x\in A'',\ s\in\big[0,\tfrac12\big], \\ 
    w_{2s-1}(x) &  &  \mbox{if } x\in A'',\ s\in\big[\tfrac12,1\big].  
  \end{array}
\right.
\qedhere
\]
\end{proof}

\subsection{The case of positive mean index}\label{s:positive_mean_index}
Let now $\gamma$ be an $I$-invariant geodesic exactly as in the previous sections, except that it has positive mean index, meaning that $\ind(E,\gamma^{mp+1})$ is not zero for some $m\in\N$. Grove and Tanaka's Lemma~\ref{l:mean_index} implies that such index tends to infinity as $m\to\infty$. In this case, the assertion of Lemma~\ref{l:iterated_mountain_passes} becomes a consequence of a general result from non-linear analysis (see e.g.\ \cite[Thm.~2.1 on page~92]{Chang:Infinite_dimensional_Morse_theory_and_multiple_solution_problems} for a similar statement in the isolated critical point case). We provide the proof in our setting of isometry-invariant geodesics, but the reader can easily extract an abstract statement that works with smooth functions having isolated critical circles.

\begin{lem}\label{l:iterated_mountain_passes_positive_index}
Let $V\subset\Lambda(M;I)$ be a sufficiently small neighborhood of an isolated critical circle $\orb(\gamma)\subset\crit(E)$. Consider a degree $d$ such that $1\leq d<\ind(E,\gamma)$, a compact domain $K\subset \R^d$, and a continuous map $u:K\to \{E<E(\gamma)\}\cup V$ such that $u(\partial K)\cap V=\varnothing$. There exists a homotopy $u_s:K\to \Lambda(M;I)$ such that $u_0=u$, $u_1(K)\subset \{E<E(\gamma)\}$, and $u_s\equiv u$ outside $u^{-1}(V)$ for all $s\in[0,1]$.
\end{lem}

\begin{proof}
The normal bundle of the critical circle $\orb(\gamma)\subset\Lambda(M;I)$ splits as the Whitney sum
\[N(\orb(\gamma))=N^0\oplus N^+\oplus N^-,\]
where the fibers of $N^0$ are the intersection of null-space of the Hessian of $E$ with the normal bundle $N(\orb(\gamma))$, while the fibers of $N^+$ and $N^-$ are the negative and positive eigenspaces respectively of the Hessian of $E$. Let 
\[\pi:N(\orb(\gamma))\to \orb(\gamma)\simeq S^1\]  
be the projection of the normal bundle onto the base. Let $V^0_R\subset N^0$, $V^+_R\subset N^+$, and $V^-_R\subset N^-$ be the open neighborhoods of the 0-section of radius $R>0$. We choose $R$ sufficiently small so that, by means of the exponential map of $\Lambda(M;I)$, we can identify $V^0_R\oplus V^+_R\oplus V^-_R$ with an open neighborhood of the critical circle $\orb(\gamma)$, and moreover there exists $\delta>0$ such that
\begin{equation}\label{e:negative_definite}
\begin{split}
&\partial_{x^-x^-}^2 E(x^0,x^+,x^-)[v,v]<-\tfrac\delta2 \|v\|_{\pi(x)}^2,\\
&\qquad\qquad\forall x=(x^0,x^+,x^-)\in V^0_R\oplus V^+_R\oplus V^-_R,\  v\in \pi|_{N^-}^{-1}(\pi(x)).
\end{split}
\end{equation}
Here, the norm on the bundle $N^-$ is the one associated with the Riemannian metric of $\Lambda(M;I)$. By the implicit function theorem, there exists $r\in(0,R/2)$ and a bundle map 
\[\phi:V^0_r\oplus V^+_r\to V^-_{R/2}\] 
such that, for each $(x^0,x^+,x^-)\in V^0_r\oplus V^+_r\oplus V^-_{R/2}$, we have $\partial_{x^-}E(x^0,x^+,x^-)=0$ if and only if $x^-=\phi(x^0,x^+)$. We use $\phi$ to define the diffeomorphism onto its image
\[\Phi:V^0_r\oplus V^+_r\oplus V^-_{R/2}\to V^0_r\oplus V^+_r\oplus V^-_R\] 
given by $\Phi(x^0,x^+,\phi(x^0,x^+)+x^-)$. From now on, we will replace the energy function $E$ by the composition $E\circ \Phi$, so that $\partial_{x^-}E(x^0,x^+,x^-)=0$ if and only if $x^-=0$. By equation~\eqref{e:negative_definite} we have
\begin{equation}\label{e:negative_margin}
 \begin{split}
E(x^0,x^+,x^-)\leq E(x^0,x^+,0)-\delta\,\|x^-\|^2_{\pi(x)},\qquad\\
\forall x=(x^0,x^+,x^-)\in V^0_r\oplus V^+_r\oplus V^-_{R/2}.
\end{split}
\end{equation}
Let $\eta_s:V^-_{R/2}\setminus\{0\mbox{-section}\}\to  \overline{V^-_{R/2}}\setminus\{0\mbox{-section}\}$, for $s\in[0,1]$, be the radial deformation 
\[\eta_s(x^-)= \left((1-s)+ s \frac{R}{2\|x^-\|_{\pi(0,0,x^-)}}\right) x^-,\]
whose time-1 map $\eta_1$ is the radial retraction onto $\partial V^-_{R/2}$. Since $x^-\mapsto E(x^0,x^+,x^-)$ is a fiberwise strictly concave function, it decreases along the deformation $\eta_s$, i.e.\ 
\[\tfrac{\diff}{\diff s} E(x^0,x^+,\eta_s(x^-))<0.\]
Now, let $\rho\in(0,r/2)$ be so small that 
\begin{align}\label{e:positive_margin}
\sup_{V^0_{\rho}\oplus V^-_{\rho}\oplus V^+_{\rho}} E < E(\gamma)+\delta R/2.
\end{align} 
We require the neighborhood $V$ of the statement to be small enough so that
\[
V\subset V^0_{\rho/2}\oplus V^-_{\rho/2}\oplus V^+_{\rho/2}.
\]
Let $u:K\to \{E<E(\gamma)\}\cup V$ be a continuous map such that $u(\partial K)\cap V=\varnothing$. We consider the open set $A:=u^{-1}(V)$. Notice that there exists a compact subset $K'\subset A$ and a quantity $\epsilon>0$ such that $u(K\setminus K')\subset\{E<E(\gamma)-\epsilon\}$. Therefore, for our purpose, we only need to focus on the restriction
\[u|_{A}=(u^0,u^+,u^-):A\to V.\] 
By our assumption, the degree $d=\dim(K)$ is smaller than the Morse index $\ind(E,\gamma)$, which is the rank of the bundle $N^-$. Therefore, by the  transversality theorem, we can perturb the map $u^-$ inside a compact subset of $A$ in order to  obtain that $u^-(K')$ does not intersect the 0-section of $N^-$. The perturbation can be chosen arbitrarily small in the $C^0$-topology, and in particular so small that the perturbed $u$ still maps the complement of $K'$ to the sublevel set $\{E<E(\gamma)\}$, and $A$ to the set $V^0_\rho\oplus V^+_\rho\oplus V^-_\rho$.

Let $\chi:K\to[0,1]$ be a smooth map that is identically 1 on $K'$ and has support inside $A$. The desired homotopy $u_s:K\to \Lambda(M;I)$ will be given by 
\[u_s(z):=(u^0(z),u^+(z),\eta_{\chi(z)s}\circ u^-(z)),\qquad\forall z\in A.\] Indeed, $u_s$ is stationary on the complement of the compact subset $\mathrm{supp}(\chi)\subset A$, it does not increase the energy, i.e.
\[
\tfrac{\diff}{\diff s} E(u_s(z))\leq0,
\]
and, by equations~\eqref{e:negative_margin} and~\eqref{e:positive_margin}, its time-1 map $u_1$ satisfies
\[
E(u_1(z))
\leq
E(u(z))  - \delta R/2
< E(\gamma),\qquad\forall z\in K'.\qedhere
\]
\end{proof}

\section{The homotopic multiplicity result}\label{s:main_theorem}

Let $(M,g)$ be a non-simply connected closed Riemannian manifold. Each connected component of its free loop space $\Lambda(M;\id)$ is not simply connected. Indeed, the fundamental group of $M$ injects into the fundamental group of the connected components of $\Lambda(M;\id)$ of the contractible loops; moreover, given any element $\gamma\in\Lambda(M;\id)$, its orbit under the circle action $t\cdot\gamma=\gamma(t+\cdot)$ is a contractible loop in $\Lambda(M;\id)$ if and only if $\gamma$ itself is a contractible loop in $M$. 

Let $I$ be an isometry of $(M,g)$. If $I$ is homotopic to the identity (not necessarily through isometries), the space of $I$-invariant curves $\Lambda(M;I)$ is homotopy equivalent to the free loop space $\Lambda(M;\id)$, see~\cite[Lemma~3.6]{Grove:Condition_C_for_the_energy_integral_on_certain_path_spaces_and_applications_to_the_theory_of_geodesics}. Therefore, all the connected components of $\Lambda(M;I)$ are not simply connected. If $M$ were the circle, all connected components of $\Lambda(M;I)$ would be homotopy equivalent to the circle. Obviously, in this case, there is only one  $I$-invariant geodesic: the circle $M$ itself. The next theorem, which is the main result of this paper, asserts that there are infinitely many isometry-invariant geodesics whenever the homotopy groups of $\Lambda(M;I)$ are richer than in the example of the circle.

\begin{thm}\label{t:main}
Let $(M,g)$ be a connected closed Riemannian manifold equipped with an isometry $I$. Assume that there exists a degree $d\geq 1$ and an infinite sequence of pairwise distinct connected components $C_n\subset\Lambda(M;I)$, for $n\in\N$, with non-trivial homotopy group $\pi_d(C_n)$. If $d=1$, assume further that these fundamental groups are not cyclic. Then, there exist infinitely many $I$-invariant geodesics in $C:=\cup_{n\in\N} C_n$.
\end{thm}

\begin{proof}
Since we are looking for infinitely many $I$-invariant geodesics in $C$, we can assume that the critical circles of $E|_C$ are isolated. In particular, all $I$-invariant geodesics contained in $C$ must be periodic curves. We choose as base-point of $C_n$ an $I$-invariant geodesic $\alpha_n$ whose critical circle is a global minimum of $E$ in the connected component. In general we cannot expect the family of the $\alpha_n$'s to give infinitely many $I$-invariant geodesics: for instance, in the worst case scenario, $\alpha_1$ may be an $I$-invariant geodesic that is periodic with minimal period $q\geq1$, and  each $\alpha_n$ may be of the form $\alpha_1^{nq+1}$.

We will prove the theorem by contradiction: we assume that $C$ contains only the $I$-invariant geodesics $\gamma_1,...,\gamma_r$ (together with infinitely many of their ``iterates''). For each of these $\gamma_i$'s, we denote by $p_i\geq1$ its minimal period, and by $\overline m_i$ the integer such that for all $m\geq\overline m_i$ the assertion of Lemma~\ref{l:iterated_mountain_passes} holds. If $n$ is large enough, say $n\geq\overline n$, an iterated $I$-invariant geodesic $\gamma_i^{mp_i+1}$ belongs to $C_n$ only if $m\geq\overline m_i$.

We fix $n\geq\overline n$ and a non-trivial homotopy class $h_n\in\pi_d(C_n,\alpha_n)$. If $d=1$, we assumed that $\pi_d(C_n,\alpha_n)$ is not cyclic, and therefore we can choose $h_n$ that is not represented by a multiple of the critical circle $\orb(\alpha_m)$. The representatives of $h_n$ are maps of the form $u:(B^d,\partial B^d)\to(C_n,\alpha_n)$. Notice that every such map cannot have image entirely contained in a small tubular neighborhood $N\subset C_n$ of $\orb(\alpha_n)$. Otherwise, the map $u$ could be deformed inside  $\orb(\alpha_n)$, which would force $d=1$ and $h_n$ to be represented by a multiple of the critical circle $\orb(\alpha_n)$. Since $\orb(\alpha_n)$ is a strict local minimum of the energy $E$, we have
\begin{align}\label{e:minimax}
c:=\inf_{[u]=h_n} \max_{x\in B^d} E(u(x))>E(\alpha_n). 
\end{align}
By Morse theory, $c$ is a critical value of $E|_{C_n}$. The assumption made in the previous paragraph implies that every critical circle in $C_n$ at level $c$ is of the form $\gamma_i^{m_ip_i+1}$, where $\gamma_i\in\{\gamma_1,...,\gamma_r\}$ and $m_i\geq \overline m_i$. We denote by $V_i\subset C_n$ a sufficiently small neighborhood of the critical circle $\orb(\gamma_i^{m_ip_i+1})$ that does not contain other critical circles of $E$, and by $V$ the (disjoint) union of the these neighborhoods for all the  critical circles in $C_n$ at level $c$. By Morse theory, we can find a representative $u:(B^d,\partial B^d)\to(C_n,\alpha_n)$ of $h_n$ such that $u(B^d)\subset \{E<c\}\cup V$. By Lemma~\ref{l:iterated_mountain_passes}, there exists a homotopy $u_s:(B^d,\partial B^d)\to(C_n,\alpha_n)$ such that $u_0=u$ and $u_1(B^d)\subset\{E<c\}$. This contradicts the definition of $c$ in~\eqref{e:minimax}.
\end{proof}

\begin{proof}[Proof of Theorem~\ref{t:pi_1_Z}]
Let $\ev:\Lambda(M;\id)\to M$ be the evaluation map $\ev(\gamma)=\gamma(0)$. We fix a generator $\theta$ of the fundamental group $\pi_1(M)\simeq\Z$, and  we denote by $D_n\subset\Lambda(M;\id)$ the connected component of the free loop space containing the $n$-fold iteration of the curve $\theta$. Since the closed manifold $M$ is not a circle, it is not homotopy equivalent to a circle. In particular, by Whitehead Theorem, $M$ must have a non-trivial homotopy group in some degree $d\geq 2$. We fix $d$ to be a minimal such degree. Bangert and Hingston proved in \cite{Bangert_Hingston:Closed_geodesics_on_manifolds_with_infinite_Abelian_fundamental_group} that there exists $k\in\N$ and, for all $n\in\N$, a non-trivial homotopy class $h_{n}\in\pi_{d-1}(D_{kn})$ such that $\ev_*(h_n)=0$. As we have already remarked at the beginning of this section, all connected components $D_n$ are not simply connected, and the evaluation map induces a surjective homomorphism $\ev_*:\pi_1(D_n)\to\pi_1(M)$ for all $n>0$. Therefore, if the above degree $d$ is equal to 2, Bangert and Hingston's result implies that, for all $n>0$, the fundamental group $\pi_{d-1}(D_{kn})$ is not cyclic.

Now, let $I$ be an isometry that is homotopic to the identity. We already recalled  that there exists a homotopy equivalence $\iota:\Lambda(M;\id)\to\Lambda(M;I)$. We denote by $C_n\subset\Lambda(M;I)$ the connected component containing $\iota(D_n)$. Notice that these connected components are pairwise distinct, and they have the same homotopy groups as the corresponding $D_n$'s. Therefore, the sequence $C_{kn}$, for $n\in\N$, satisfies the assumptions of Theorem~\ref{t:main}, which implies that the union $C:=\cup_{n\in\N} C_{kn}$ contains infinitely many $I$-invariant geodesics.
\end{proof}

\bibliography{_biblio}
\bibliographystyle{amsalpha}

\end{document}